\documentclass[11pt, a4paper]{amsart}
\usepackage{amsfonts,amssymb,amsmath,amsthm,amscd,mathtools,multicol,tikz, tikz-cd,caption,enumerate,mathrsfs,thmtools,cite,enumitem, float}
\usetikzlibrary{shapes, arrows.meta, positioning}
\usepackage{inputenc}
\usepackage[foot]{amsaddr}
\usepackage[pagebackref=true, colorlinks, linkcolor=blue, citecolor=red]{hyperref}
\usepackage{latexsym}
\usepackage{fullpage}
\usepackage{microtype}
\usepackage{subfiles} 
\setlist{topsep=0.5em, itemsep=1em} 

\usepackage{palatino}
\makeatletter
\makeindex 
\newcommand{\be}{\begin{equation}}
\newcommand{\ee}{\end{equation}}
\newcommand{\beano}{\begin{eqn*}} 
	\newcommand{\eeano}{\end{eqnarray*}}
\newcommand{\ba}{\begin{array}}
	\newcommand{\ea}{\end{array}}

\declaretheoremstyle[headfont=\normalfont]{normalhead}

\newtheorem{theorem}{Theorem}[section]
\newtheorem{theoremalph}{Theorem}[section]

\newtheorem{lemma}[theorem]{Lemma}
\newtheorem{corollary}[theorem]{Corollary}

\newtheorem{proposition}[theorem]{Proposition}
\newtheorem{remark}[theorem]{Remark}

\theoremstyle{definition}
\newtheorem{example}{Example}[section]

\numberwithin{equation}{section}

\begin{document}
\title{Defects in weighted graphs and Commutators}
\author{Harish Kishnani}
\email{harishkishnani11@gmail.com}
\address{Indian Institute of Science Education and Research Mohali, Sector 81, Mohali 140306, India}

\author{Amit Kulshrestha}
\email{amitk@iisermohali.ac.in}
\address{Indian Institute of Science Education and Research Mohali, Sector 81, Mohali 140306, India}

\thanks{The authors thank the DST-FIST facility established through grant SR/FST/MS-I/2019/46 to support this research. The first named author acknowledges the support through the Prime Minister's Research Fellowship from the Ministry of Education, Government of India (PMRF ID: 0601097).}

\begin{abstract}
Let $R$ be a commutative ring. In \cite{KK_2025(1)}, the authors introduced $R$-weighted graphs as a tool for studying commutators in groups and Lie algebras. These graphs are equivalent to a system of balance equations, and their consistent labelings correspond to solutions of this system of balance equations. In this article, we apply these ideas in the case when $R$ is a field $F$. We focus on $F$-weighted graphs with four vertices and establish necessary and sufficient conditions for the existence of consistent labelings on them. A notion of defects in weighted graphs is introduced for this purpose. We prove that defects in weighted graphs prevent Lie brackets from being surjective onto its derived Lie subalgebra. Similarly, these defects prevent certain elements in the commutator subgroup of a nilpotent group of class $2$ from being a commutator.
As an application of our techniques, we prove that for a Lie algebra $L$ whose dimension over $F$ is at most countable and the dimension of its derived subalgebra $L'$ is at most $3$, the Lie bracket is surjective onto 
$L'$. We provide a counterexample when $\dim(L') = 4$. We also characterize commutators among $L$' for the Lie algebras $L$ with $\dim(L/Z(L))\leq 4$.
\end{abstract}
\maketitle

\section{Introduction}
Let $F$ be a field, and let $\Gamma := (V, E)$ be a graph with vertex set $V := \{v_1, v_2, \dots, v_n\}$, where the subscripts induce a complete ordering on $V$. For $i < j$, denote by $e_{i,j}$ the edge between $v_i$ and $v_j$, whenever it exists. We assign $F$-valued weights to the edges of $\Gamma$ by equipping the edge $e_{i,j}$ with weight $d_{i,j} \in F$. Such graphs are called \textit{weighted ordered simple graphs} over $F$. The sequence $(d_{i,j})_{e_{i,j} \in E}$ is called the \textit{weight sequence} of $\Gamma$. Throughout this article, by a \textit{weighted graph} we mean a weighted ordered simple graph over $F$.
When the reference to weights is irrelevant, we refer to these graphs as \textit{ordered simple graphs}, or simply a graph. A vertex $v_i$ of a weighted graph $\Gamma$ is said to be a \textit{null vertex} if the weight of each edge incidenting at $v_i$ is $0$.

A \textit{labeling} of a weighted graph $\Gamma = (V, E)$ with weight sequence $(d_{ij})_{e_{i,j} \in E}$ is an assignment  
$v_k \mapsto (a_k, b_k) \in F \times F$, for each vertex $v_k \in V$. A labeling $v_k \mapsto (a_k, b_k)$ is \textit{consistent} if $a_i b_j - a_j b_i = d_{i,j}$ for every $i < j$ satisfying $e_{i,j} \in E$.
There are weighted graphs which do not admit a consistent labeling. Let $A(n) := \{(i,j) : 1 \leq i < j \leq n\},$ 
and $A_E$ denote the indexing set of edges in $\Gamma$. Then $A_E$ is a subset of $A(n)$. The weight sequence may be thought of as a function $D : A_E \to F$ that maps the index $(i,j)$ to $d_{i,j}$.
For each $(i,j) \in A_E$, the equation $x_i y_j - x_j y_i = d_{i,j}$ is called a \textit{balance equation}.
Solving the system of these balance equations, as $(i,j)$ varies over $A_E$, is equivalent to finding a consistent labeling on $\Gamma$.

The study of commutators, in both $p$-groups and Lie algebras, via balance equations and weighted graphs was initiated in~\cite{KK_2025(1)}. In that framework, graph theoretic properties were examined to determine if a product or linear combination of commutators is itself a commutator. An obstruction that prevents elements from being a commutator, or equivalently, the corresponding graph from admitting a consistent labeling, was formulated in terms of bad cycles in graphs \cite[Theorem 3.4]{KK_2025(1)}. In this article, we focus on weighted graphs with four vertices and establish necessary and sufficient conditions for the existence of consistent labelings of $F$-weighted graphs. The obstructions to a consistent labeling are presented as certain defects of the graph. Formulation of defects for graphs with four vertices is the content of \S\ref{Defects in graphs having 4 vertices}. In \S\ref{Labeling of graphs with four vertices}, we prove that among the weighted graphs with four vertices, those admitting a consistent labeling are precisely the defectless ones.

We devote \S\ref{section on bilinear maps and Lie algebra} to apply these results to establish an analogue of L'vov-Kaplansky Conjecture, in the set up of Lie algebras whose dimension is at most countable.

Let $L$ be a Lie algebra over a field $F$. Let $[L,L]$ be the set of commutators of $L$. That is,
$$[L,L] = \{[x,y] : x,y \in L\}.$$

Let $L'$ be the set of linear combinations of elements from $[L,L]$. It is the \emph{derived Lie algebra} of $L$. A question that is analogous to L'vov-Kaplansky Conjecture is to determine if $L' = [L,L]$. In \cite{Brown_1963}, Brown proved that $[L,L] = L$ for any finite-dimensional complex simple Lie algebra. Akhiezer extended this result to most of the finite-dimensional simple real Lie algebras \cite{Akhiezer_2015}, but the problem is still open for arbitrary finite-dimensional simple real Lie algebras. For a finite-dimensional nilpotent Lie algebra with $\dim (L') \leq 4$, this problem has been investigated in \cite{Niranjan-Rani_2023}. In \cite{DKR_2021} and \cite{KMR_2024}, the authors have given examples of infinite-dimensional simple Lie algebras with $[L,L] \neq L$. The image of Lie polynomials on real Lie algebras of dimension up to 3 is studied in \cite{CF-24}. The existing results are often restricted to a type of Lie algebra, or have restrictions concerning the field. Our results do not have such restrictions.

\begin{theoremalph}
Let $F$ be a field and $L$ be a Lie algebra over $F$ with at most countable dimension. If $\dim(L') \leq 3$, then $L'=[L,L]$. (Theorem \ref{Whole space Lie algebra dim 3})
\end{theoremalph}

When $\dim(L') \geq 4$, our techniques allow us to construct counterexamples to the analogue of L'vov-Kaplansky Conjecture
by enabling precise characterization of elements in $L'$ which are commutators. For $x \in L'$, a presentation of $x$ is the weight sequence $D = (d_{i,j})_{e_{i,j} \in E}$ of a weighted graph that is constructed through the process described in the beginning of \S\ref{section on bilinear maps and Lie algebra}. We analyze these weighted graphs for defects to determine if $x$ is a commutator.

\begin{theoremalph}
Let $F$ be a field and $L$ be a Lie algebra over $F$ with at most countable dimension. Let $x \in L'$ be such that for some presentation $D$ of $x$, the corresponding weighted graph $\Gamma$ has four vertices and is defectless. Then $x \in [L,L]$. (Theorem \ref{Elementwise Lie algebra without defects theorem})
\end{theoremalph}

Conversely, we have the following result.

\begin{theoremalph}
Let $F$ be a field and $L$ be a Lie algebra over $F$ with at most countable dimension. Let $x \in L'$ be such that for each presentation $D$ of $x$, the corresponding weighted graph $\Gamma$ contains a defect of type $(m)_A$, $(4)_B$ or $(4)_C$. Then $x \notin [L,L]$. (Theorem \ref{Elementwise Lie algebra defects theorem})
\end{theoremalph}

We emphasize that our techniques are not restricted to Lie brackets and they work for any alternating bilinear map. Further, in \S\ref{section for results on p-groups} we deploy these techniques to provide results on the commutators in $p$-group of nilpotency class $2$.
We denote by $K(G)$ the set of commutators in $G$, and by $G'$ the commutator subgroup of $G$.
For $g \in G'$ we construct presentations and corresponding weighted graphs in \S\ref{section for results on p-groups}, and obtain the following results. For a detailed discussion on the significance of these results, we refer to \cite{Kappe_Morse_2007}. In \cite{KK_2025(1)} too, we have addressed such questions.

\begin{theoremalph}
Let $G$ be a $p$-group of nilpotency class $2$ and exponent $p$. Let $g \in G'$ be such that for some presentation of $g$, the corresponding weighted graph $\Gamma$ has four vertices and is defectless. Then $g \in K(G)$. (Theorem \ref {Elementwise group without defects theorem})
\end{theoremalph}

\begin{theoremalph}
Let $G$ be a $p$-group of nilpotency class $2$. Let $g \in G'$ be such that for all presentations 
of $g$, the corresponding weighted graph $\Gamma$ contains a defect of type $(m)_A$, $(4)_B$ or $(4)_C$. Then $g \notin K(G)$. (Theorem \ref{Elementwise group defects theorem})
\end{theoremalph}

\section{Defects in graphs with four vertices}\label{Defects in graphs having 4 vertices}
In this section, we characterize obstructions to consistent labelings and call them defects in weighted graphs.

Two graphs $\Gamma = (V, E)$ and $\Gamma' = (V', E')$ over a field $F$ are \textit{isomorphic} if there exists a bijection $\varphi \colon V \to V'$ preserving adjacency, \textit{i.e.} vertices $v_i$ and $v_j$ are adjacent in $\Gamma$ if and only if $\varphi(v_i)$ and $\varphi(v_j)$ are adjacent in $\Gamma'$. Such a map
$\varphi$ is called an \textit{isomorphism} of graphs, and we write $\Gamma \simeq \Gamma'$ to denote that the two graphs are isomorphic.

Two weighted graphs $\Gamma = (V, E)$ and $\Gamma' = (V', E')$, with vertex sets $V = \{v_1, \dots, v_n\}$ and $V' = \{v'_1, \dots, v'_n\}$, and weight sequences $(d_{i,j})_{e_{i,j} \in E}$ and $(d'_{i,j})_{e'_{i,j} \in E'}$ respectively, are called \textit{weighted isomorphic} if there exists a bijection $\varphi : V \to V'$ which is a graph isomorphism between $\Gamma$ and $\Gamma'$ satisfying
$$
d_{i,j} = \begin{cases} d'_{k,\ell}, \quad \quad \text{ if } v'_k = \varphi(v_i) < \varphi(v_j) = v'_{\ell},\\ 
 -d'_{\ell,k}, ~~~~\quad \text{ if } v'_k = \varphi(v_i) > \varphi(v_j) = v'_{\ell}. 
\end{cases}
$$

We denote $\Gamma \simeq_{w} \Gamma'$ to express that the two weighted graphs are isomorphic.

A weighted graph $\Gamma_1 = (V_1, E_1)$ is called a \textit{subgraph} of a weighted graph $\Gamma_2 = (V_2, E_2)$ if
\begin{enumerate}
\item[(i).] $V_1 \subseteq V_2$ and $E_1 \subseteq E_2$.
\item[(ii).] For every edge $e_{i,j} \in E_1$, its weight in $\Gamma_1$ equals its weight in $\Gamma_2$.
\end{enumerate}

We recall the definitions of bad cycles and unfavorable proximity from \cite{KK_2025(1)}.
A simple cycle $C_r: v_{i_1} \to v_{i_2} \to \dots \to v_{i_r} \to v_{i_1}$ in a graph $\Gamma$ is a \textit{bad cycle} if $\deg_{\Gamma}(v_{i_j}) > 2$ for $j \leq r-2$; and $\deg_{\Gamma}(v_{i_{r-1}})= \deg_{\Gamma}(i_{i_r}) = 2$. For each vertex $v_{i_j}$, $1 \leq j \leq r-2$ a bad cycle $C_r$, we pick an adjacent vertex $v_{i_j+r}$ in $\Gamma$ that lies outside $C_r$. The subgraph $\mathcal{P}$ of $\Gamma$ containing $C_r$, vertices $v_{i_j+r}$, and the edges between $v_{i_j}$ and $v_{i_j+r}$ is called a \textit{proximity} for $C_r$ in $\Gamma$. The proximity $\mathcal{P}$ is said to be \textit{unfavorable} if $d_{i_1, i_2} = d_{i_2, i_3} = \cdots = d_{i_{r-2}, i_{r-1}} = d_{i_1, i_r} = 0$ and $d_{i_{r-1}, i_r}$, $d_{i_1, i_{r+1}}$, $d_{i_2, i_{r+2}}, \dots, d_{i_{r-2}, i_{2r-2}} \neq 0$.

{\bfseries Defect of type $(m)_A$}. A weighted graph $\Gamma$ is said to have a \emph{defect of type $(m)_{A}$} if it contains a weighted subgraph that is weighted isomorphic to an unfavorable proximity $\mathcal P$ for a bad cycle in $\frac{1}{2}(m+2)$ vertices. We note that $m$ is necessarily an even integer which is equal to the number of vertices in $\mathcal P$. The subscript $A$ signifies the nature of restrictions on weights of $\Gamma$, that obstructs it from admitting a consistent labeling \cite[Theorem 3.4]{KK_2025(1)}.

\begin{lemma}[Theorem 3.4,\cite{KK_2025(1)}]\label{defect (m)_A Lemma}
Let $\Gamma$ be a weighted graph with a defect of type $(m)_A$. Then $\Gamma$ does not admit a consistent labeling.
\end{lemma}

The smallest value of $m$ for which a graph may possess a defect of type $(m)_{A}$ is $4$. In this article, our focus is on weighted graphs with four vertices. The graph $\Gamma_1$ in Figure \ref{defect (4)_A} has a defect of type $(4)_A$.
\begin{center}
\begin{figure}[h]
\begin{tikzpicture}
  [scale=0.8,auto=left,every node/.style={circle,fill=blue!20}]
  \node (n3) at (9,1)      {$v_3$};  
  \node (n1) at (5,5)      {$v_1$};
  \node (n2) at (1,1)      {$v_2$};
  \node (n4) at (13,5)     {$v_4$}; 
  \draw (n1) -- node[midway, fill = none] {$0$} (n3);  
  \draw (n2) -- node[midway, below, fill = none] {$0$} (n3);
  \draw (n1) -- node[midway, left, fill = none] {$\neq 0$} (n2);  
  \draw (n3) -- node[midway, left, fill = none] {$\neq 0$} (n4);
  \foreach \from/\to in {n2/n3,n1/n3,n1/n2,n3/n4}
      \draw (\from) -- (\to);
\end{tikzpicture}
\caption{{\bfseries Defect $(4)_A$}: Graph $\Gamma_1$ with $d_{1,2} \ne 0$, $d_{3,4} \ne 0$ and $d_{1,3}=d_{2,3}=0$.}\label{defect (4)_A}
\end{figure}
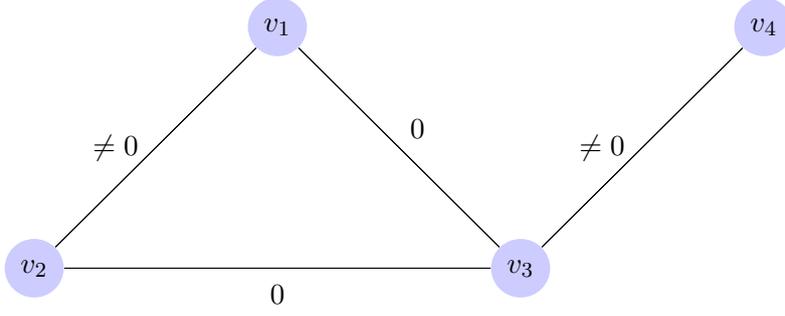
\end{center}


We remark that a consistent labeling of a weighted graph passes onto a consistent labeling of its weighted
subgraphs. Thus, if a weighted subgraph of a weighted graph $\Gamma$ does not admit a consistent labeling, then
$\Gamma$ itself does not admit a consistent labeling.

{\bfseries Defect of type $(4)_B$}. A weighted graph $\Gamma$ is said to have a \emph{defect of type $(4)_B$} if it contains a weighted subgraph that is weighted isomorphic to the graph $\Gamma_2$ in Figure \ref{defect (4)_B}.
\begin{center}
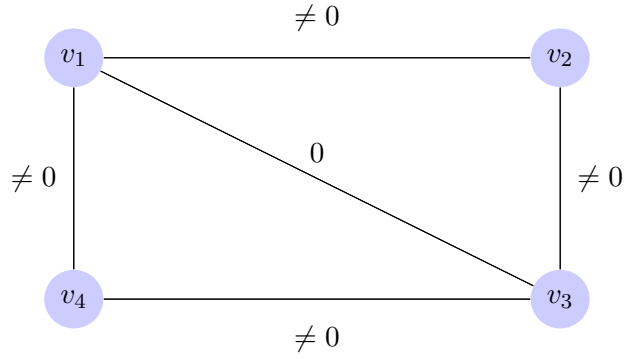
\begin{figure}[h]
\begin{tikzpicture}
  [scale=0.8,auto=left,every node/.style={circle,fill=blue!20}]
  \node (n3) at (9,1)      {$v_3$};
  \node (n1) at (1,5)      {$v_1$};
  \node (n4) at (1,1)      {$v_4$};
  \node (n2) at (9,5)     {$v_2$}; 
  \draw (n1) -- node[midway, fill = none] {$\neq 0$} (n2);  
  \draw (n1) -- node[midway, left, fill = none] {$\neq 0$} (n4);
  \draw (n2) -- node[midway, right, fill = none] {$\neq 0$} (n3);  
  \draw (n3) -- node[midway, below, fill = none] {$\neq 0$} (n4);
  \draw (n1) -- node[midway, above, fill = none] {$0$} (n3);
  \foreach \from/\to in {n2/n3,n1/n3,n1/n4,n3/n4,n1/n2}
      \draw (\from) -- (\to);
\end{tikzpicture}
\caption{{\bfseries Defect $(4)_B$}: Graph $\Gamma_2$ with $d_{1,2}d_{1,4}d_{2,3}d_{3,4} \ne 0$, $d_{1,3}=0$ and $d_{1,2}d_{3,4} \ne -d_{1,4}d_{2,3}$.}\label{defect (4)_B}
\end{figure}
\end{center}

\begin{lemma}\label{defect (4)_B Lemma}
Let $\Gamma$ be a weighted graph with a defect of type $(4)_B$. Then $\Gamma$ does not admit a consistent labeling.
\end{lemma}

\begin{proof}
It is enough to prove the lemma for the graph $\Gamma_2$ in Figure \ref{defect (4)_B}. 
We assume a consistent labeling $v_k \mapsto (\alpha_k, \beta_k)$ on $\Gamma_2$. Since $d_{1,2} \neq 0$, it is clear that $(\alpha_1, \beta_1) \neq (0,0)$. Suppose $\beta_1 \neq 0$ (the case $\alpha_1 \neq 0$ will lead to the same conclusion).
This, together with the restriction $d_{1,3} =0$ yields
\begin{align*}
&0=\beta_2 d_{1,3}\\
&= \beta_2(\alpha_1\beta_3-\alpha_3\beta_1)\\
&=\beta_3(\alpha_1\beta_2-\alpha_2\beta_1)+\beta_1(\alpha_2\beta_3-\alpha_3\beta_2)\\
&=\beta_3 d_{1,2}+\beta_1 d_{2,3}
\end{align*}

Thus, $\beta_3\beta_1^{-1} = -d_{1,2}^{-1} d_{2,3}$. Similarly, $0=\beta_4 d_{1,3} = \beta_3 d_{1,4} -\beta_1 d_{3,4}$ and thus $\beta_3\beta_1^{-1} = d_{1,4}^{-1} d_{3,4}$. Equating the two values of $\beta_3\beta_1^{-1}$ we obtain $d_{1,2}d_{3,4} = -d_{1,4}d_{2,3}$. This shows that if $\Gamma$ has a defect of type $(4)_B$, then it does not admit a consistent labeling.
\end{proof}

{\bfseries Defect of type $(4)_C$}. A weighted graph $\Gamma$ is said to have a \emph{defect of type $(4)_C$} if it contains a weighted subgraph that is weighted isomorphic to the graph $\Gamma_3$ in Figure \ref{defect (4)_C}.

\begin{center}
\begin{figure}[h]
\begin{tikzpicture}
  [scale=0.8,auto=left,every node/.style={circle,fill=blue!20}]
  \node (n3) at (9,1)      {$v_3$};
  \node (n1) at (1,5)      {$v_1$};
  \node (n4) at (1,1)      {$v_4$};
  \node (n2) at (9,5)     {$v_2$}; 
  \draw (n1) -- node[midway, fill = none] {$\neq 0$} (n2);  
  \draw (n1) -- node[midway, left, fill = none] {$\neq 0$} (n4);
  \draw (n2) -- node[midway, right, fill = none] {$\neq 0$} (n3);  
  \draw (n3) -- node[midway, below, fill = none] {$\neq 0$} (n4);
  \draw (n1) -- node[pos = 0.3, above, yshift = -7 pt, fill = none] {$\neq 0$} (n3);  
  \draw (n2) -- node[pos = 0.7, below, yshift = 7 pt, fill = none] {$\neq 0$} (n4);  
  \foreach \from/\to in {n2/n3,n1/n3,n1/n4,n3/n4,n1/n2,n2/n4}
      \draw (\from) -- (\to);
\end{tikzpicture}
\caption{{\bfseries Defect $(4)_C$}: Graph $\Gamma_3$ with $d_{1,2}d_{1,3}d_{1,4}d_{2,3}d_{2,4}d_{3,4} \ne 0$ and $d_{1,2}d_{3,4} + d_{1,4}d_{2,3} \ne d_{1,3}d_{2,4}$.}\label{defect (4)_C}
\end{figure}
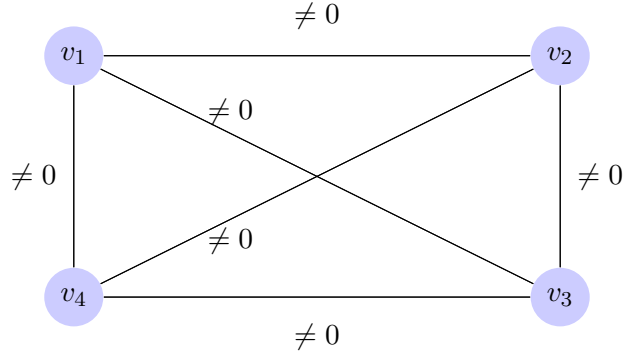
\end{center}

\begin{lemma}\label{defect (4)_C Lemma}
Let $\Gamma$ be a weighted graph with a defect of type $(4)_C$. Then $\Gamma$ does not admit a consistent labeling.
\end{lemma}

\begin{proof}
It is enough to prove the lemma for the graph $\Gamma_3$ in Figure \ref{defect (4)_C}.
Let $v_k \mapsto (\alpha_k, \beta_k)$ be a consistent labeling on $\Gamma_3$. Since $d_{1,2} \neq 0$, it is clear that $(\alpha_1, \beta_1) \neq (0,0)$. Suppose $\beta_1 \neq 0$ (the case $\alpha_1 \neq 0$ will lead to the same conclusion). Then $\beta_1 d_{2,3} = \beta_1(\alpha_2\beta_3-\alpha_3\beta_2) = \beta_2d_{1,3}-\beta_3d_{1,2}$. 

Similarly, $\beta_1 d_{2,4} = \beta_2d_{1,4}-\beta_4d_{1,2}$ and $\beta_1 d_{3,4} = \beta_3d_{1,4}-\beta_4d_{1,3}$. 
Hence, 
\begin{align*}
\beta_1(d_{1,2}d_{3,4} + d_{2,3}d_{1,4}) &= d_{1,2}(\beta_1 d_{3,4})+d_{1,4}(\beta_1d_{2,3}) \\
&=d_{1,2}(\beta_3d_{1,4}-\beta_4d_{1,3})+d_{1,4}(\beta_2d_{1,3}-\beta_3d_{1,2})\\
&=-d_{1,2}\beta_4d_{1,3}+d_{1,4}\beta_2d_{1,3}\\
&=d_{1,3}(\beta_2d_{1,4}-\beta_4d_{1,2})\\
&=d_{1,3}\beta_1 d_{2,4} 
\end{align*}
Since $\beta_1 \ne 0$, we have $d_{1,2}d_{3,4} + d_{2,3}d_{1,4} = d_{1,3}d_{2,4}$.
This shows that if $\Gamma$ has a defect of type $(4)_C$, then it does not admit a consistent labeling.
\end{proof}

\section{Consistent labeling of graphs with four vertices}\label{Labeling of graphs with four vertices}
Recall that a weighted graph $\Gamma$ with at most $3$ vertices is either a tree or a triangle. Thus, by \cite[Lemma 3.1]{KK_2025(1)} and \cite[Lemma 3.7]{KK_2025(1)}, it admits a consistent labeling. 
To address the problem of consistent labeling, it is enough to consider connected weighted graphs without null vertices. We carry this assumption throughout this section, and examine graphs with four vertices for consistent labeling.

A weighted graph with four vertices is \emph{defective} if it has a defect of type $(4)_A$, $(4)_B$ or $(4)_C$.

\begin{theorem}\label{4 vertices classification}
Let $\Gamma$ be a weighted graph with four vertices. Then it does not admit a consistent labeling if and only if it is defective.
\end{theorem}

\begin{proof}
If $\Gamma$ is defective, then by Lemmas \ref{defect (m)_A Lemma} --  \ref{defect (4)_C Lemma}, it does not admit a consistent labeling. We assume that $\Gamma$ is not defective and show that it admits a consistent labeling.

A connected graph in four vertices is isomorphic to one of the graphs in Figure \ref{4 vertices graphs table}. If $\Gamma$
is a tree, that is, $\Gamma \simeq \Gamma_5$ or $\Gamma_6$, then by Lemma 3.1 of \cite{KK_2025(1)}, it has a consistent labeling. Further, if $\Gamma$ is a cycle, that is, $\Gamma \simeq \Gamma_4$, then by Corollary $3.6$ of \cite{KK_2025(1)}, it admits a consistent labeling. It remains to address the cases when $\Gamma \simeq \Gamma_1, \Gamma_2$ or $\Gamma_3$.
Note that these are the cases where defects may arise.

\begin{figure}
\begin{tabular}{ccccc}
\begin{tikzpicture}
  [scale=0.45,auto=left,every node/.style={circle,fill=blue!20}]
  \node (n3) at (9,1)      {$v_3$};
  \node (n1) at (1,5)      {$v_1$};
  \node (n4) at (1,1)      {$v_4$};
  \node (n2) at (9,5)     {$v_2$}; 
 \foreach \from/\to in {n1/n2,n1/n3,n2/n3,n3/n4}
      \draw (\from) -- (\to);
\end{tikzpicture}
&
\hspace{0.5cm}
&
\begin{tikzpicture}
  [scale=0.45,auto=left,every node/.style={circle,fill=blue!20}]
  \node (n3) at (9,1)      {$v_3$};
  \node (n1) at (1,5)      {$v_1$};
  \node (n4) at (1,1)      {$v_4$};
  \node (n2) at (9,5)     {$v_2$}; 
  \foreach \from/\to in {n2/n3,n1/n3,n1/n4,n3/n4,n1/n2}
      \draw (\from) -- (\to);
      
\end{tikzpicture}
&
\hspace{0.5cm}
&
\begin{tikzpicture}
  [scale=0.45,auto=left,every node/.style={circle,fill=blue!20}]
  \node (n3) at (9,1)      {$v_3$};
  \node (n1) at (1,5)      {$v_1$};
  \node (n4) at (1,1)      {$v_4$};
  \node (n2) at (9,5)     {$v_2$}; 
  \foreach \from/\to in {n2/n3,n1/n3,n1/n4,n3/n4,n1/n2,n2/n4}
      \draw (\from) -- (\to);
\end{tikzpicture}\\

$\Gamma_1$ & \hspace{0.5cm} &  $\Gamma_2$ & \hspace{0.5cm} & $\Gamma_3$\\
\end{tabular}
\end{figure}

\begin{figure}
\begin{tabular}{ccccc}
\begin{tikzpicture}
  [scale=0.45,auto=left,every node/.style={circle,fill=blue!20}]
  \node (n3) at (9,1)      {$v_3$};
  \node (n1) at (1,5)      {$v_1$};
  \node (n4) at (1,1)      {$v_4$};
  \node (n2) at (9,5)     {$v_2$}; 
 \foreach \from/\to in {n2/n3,n1/n2,n1/n4,n3/n4}
      \draw (\from) -- (\to);
\end{tikzpicture}
&
\hspace{0.5cm}
&
\begin{tikzpicture}
  [scale=0.45,auto=left,every node/.style={circle,fill=blue!20}]
  \node (n3) at (9,1)      {$v_3$};
  \node (n1) at (1,5)      {$v_1$};
  \node (n4) at (1,1)      {$v_4$};
  \node (n2) at (9,5)     {$v_2$}; 
  \foreach \from/\to in {n2/n3,n3/n4,n1/n2}
      \draw (\from) -- (\to);
      
\end{tikzpicture}
&
\hspace{0.5cm}
&
\begin{tikzpicture}
  [scale=0.45,auto=left,every node/.style={circle,fill=blue!20}]
  \node (n3) at (9,1)      {$v_3$};
  \node (n1) at (1,5)      {$v_1$};
  \node (n4) at (1,1)      {$v_4$};
  \node (n2) at (9,5)     {$v_2$}; 
  \foreach \from/\to in {n1/n3,n1/n4,n1/n2}
      \draw (\from) -- (\to);
\end{tikzpicture}\\

$\Gamma_4$ & \hspace{0.5cm} &  $\Gamma_5$ & \hspace{0.5cm} & $\Gamma_6$\\
\end{tabular}
\caption{Connected graphs in $4$ vertices}
\label{4 vertices graphs table}
\end{figure}
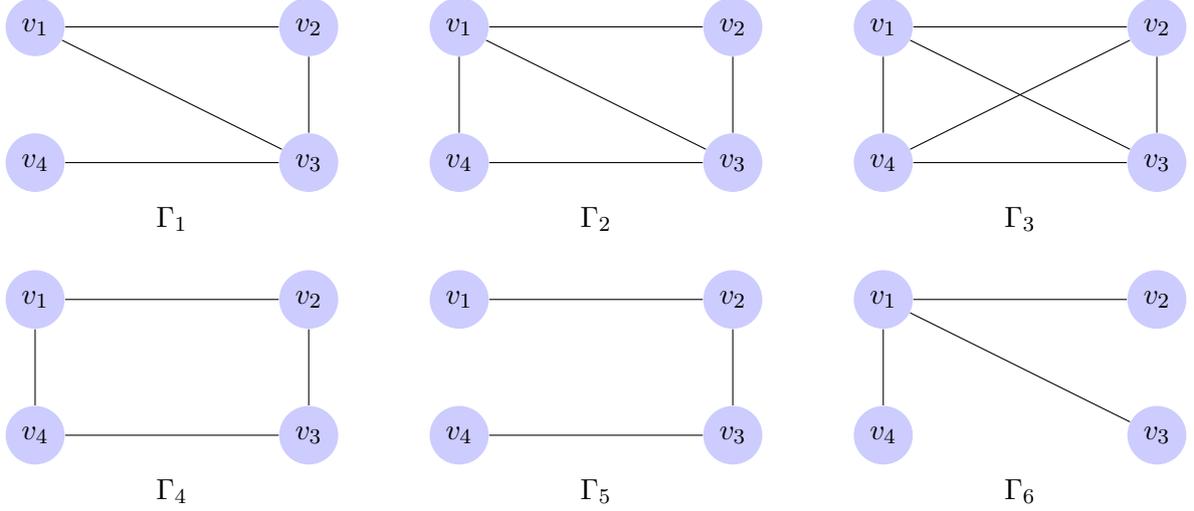

\noindent\textbf{Case 1.} \emph{$\Gamma \simeq \Gamma_1$}. \quad Since $\Gamma$ does not have null vertices, $d_{3,4} \neq 0$.

We claim that at least one of $d_{1,3}$ and $d_{2,3}$ is nonzero. Otherwise, if assume $d_{1,3} = 0$ and $d_{2,3} = 0$,
then we must have $d_{1,2} \neq 0$, for $v_1$ is not a null vertex. This exhibits a defect of type $(4)_A$ in $\Gamma$ and contradicts the assumption that $\Gamma$ is not defective.

Suppose, $d_{1,3} \neq 0$. The Table \ref{Case 4,1,1} provides a consistent labeling on $\Gamma$. The case when $d_{2,3} \neq 0$ can be addressed similarly to obtain a consistent labeling on $\Gamma$.

\begin{table}[h]
\parbox{.45\linewidth}{
\centering
\begin{tabular}{|c|c|c|}
 \hline
 $i$ &  $\alpha_i$ &$\beta_i$\\
 \hline
 $1$ & $d_{1,3}\beta_3^{-1}$ & $\beta_1 \in F$\\
 \hline
 $2$ & $d_{2,3}\beta_3^{-1}$ & $(d_{1,2}+d_{2,3}\beta_3^{-1}\beta_1)d_{1,3}^{-1}\beta_3$\\
 \hline
$3$ & $0$ & $\beta_3 \in F^{\times}$\\
 \hline
$4$ & $-d_{3,4}\beta_3^{-1}$ & $\beta_4 \in F$\\
 \hline
\end{tabular}
\caption{}
\label{Case 4,1,1}
}
\hfill
\parbox{.45\linewidth}{
\centering
 \begin{tabular}{|c|c|c|}
 \hline
 $i$ &  $\alpha_i$ &$\beta_i$\\
 \hline
 $1$ & $d_{1,4}\beta_4^{-1}$ & $0$\\
 \hline
 $2$ & $0$ & $d_{1,2}d_{1,4}^{-1}\beta_4$\\
 \hline
$3$ & $d_{3,4}\beta_4^{-1}$ & $0$\\
 \hline
$4$ & $\alpha_4 \in F$ & $\beta_4 \in F^{\times}$\\
 \hline
\end{tabular}
\caption{}
\label{Case 4,2,1}
}
\end{table}

\noindent\textbf{Case 2.} \emph{$\Gamma \simeq \Gamma_2$}. \quad If $d_{1,3} \ne 0$, then Table \ref{Case 4,1,1}
provides a consistent labeling $v_k \mapsto (\alpha_k, \beta_k)$ on $\Gamma$, with $\beta_4 = (d_{1,4} + d_{3,4}\beta_1\beta_3^{-1})d_{1,3}^{-1}\beta_3$. Thus, we assume $d_{1,3} = 0$.

We claim that $d_{1,2} \neq 0$. Otherwise, if $d_{1,2} =0$, then $d_{2,3} \neq 0$, for $v_2$ is not a null vertex. Similarly, since $v_1$ is not a null vertex, $d_{1,4} \neq 0$. Thus, we have $d_{1,3} = 0$, $d_{1,2} =0$, $d_{2,3} \neq 0$ and $d_{1,4} \neq 0$. This exhibits a defect of type $(4)_A$ in $\Gamma$ with respect to the bad cycle $v_1 \to v_2 \to v_3 \to v_1$, and contradicts the assumption that $\Gamma$ is not defective. Thus, $d_{1,2} \neq 0$.

By a similar argument, remaining weights $d_{2,3}$, $d_{1,4}$ and $d_{3,4}$ are necessarily nonzero. To avoid a defect of type $(4)_B$, we must assume $d_{1,2}d_{3,4} = -d_{2,3}d_{1,4}$. In this case, then Table \ref{Case 4,2,1} provides a consistent labeling on $\Gamma$.

\begin{table}[h]
\parbox{.45\linewidth}{
\centering
\begin{tabular}{|c|c|c|}
 \hline
 $i$ &  $\alpha_i$ &$\beta_i$\\
 \hline
 $1$ & $0$ & $\beta_1 \in F^{\times}$\\
 \hline
 $2$ & $-d_{1,2}\beta_1^{-1}$ & $d_{2,4}d_{1,4}^{-1}\beta_1$\\
 \hline
$3$ & $-d_{1,3}\beta_1^{-1}$ & $d_{3,4}d_{1,4}^{-1}\beta_1$\\
 \hline
$4$ & $-d_{1,4}\beta_1^{-1}$ & $0$\\
 \hline
\end{tabular} 
\caption{}
\label{Case 4,3,1}
}
\hfill
\parbox{.45\linewidth}{
\centering
\begin{tabular}{|c|c|c|}
 \hline
 $i$ &  $\alpha_i$ &$\beta_i$\\
 \hline
 $1$ & $0$ & $\beta_1 \in F^{\times}$\\
 \hline
 $2$ & $0$ & $d_{2,3}d_{1,3}^{-1}\beta_1$\\
 \hline
$3$ & $-d_{1,3}\beta_1^{-1}$ & $\beta_1$\\
 \hline
$4$ & $0$ & $-d_{3,4}d_{1,3}^{-1}\beta_1$\\
 \hline
\end{tabular}
\caption{}
\label{Case 4,3,2}
}
\end{table}

\noindent\textbf{Case 3}: \emph{$\Gamma \simeq \Gamma_3$}. \quad Let $|E_0|$ denote the
number of edges in $\Gamma$ whose weight is $0$. Our subcases depend on $|E_0|$.

\noindent{(a).} $|E_0| = 0$. \quad  To avoid a defect of type $(4)_C$, the weights must satisfy the relation
$$d_{1,2}d_{3,4} + d_{2,3}d_{1,4} = d_{1,3}d_{2,4}.$$

Under this condition, the Table \ref{Case 4,3,1} provides a consistent labeling of $\Gamma$.

\noindent{(b).} $|E_0| = 1$. \quad Since $\Gamma$ is a complete graph, without loss of generality, we assume that $d_{1,3}= 0$. Now, to avoid a defect of type $(4)_B$, the weights must satisfy the relation
$$d_{1,2}d_{3,4} + d_{2,3}d_{1,4} = 0.$$

Under this condition, the Table \ref{Case 4,2,1} provides a consistent labeling of $\Gamma$, with $\alpha_4 = -d_{2,4}d_{1,2}^{-1}d_{1,4}{\beta_4}^{-1}$.

\noindent{(c).} $|E_0| = 2$. \quad We claim that the two edges with zero weight do not share a common vertex. 
Suppose, to the contrary, that both edges with zero weight are incident to a common vertex. Without loss of generality, assume
$d_{1,2} = d_{1,3} = 0$ and $d_{1,4}d_{2,3}d_{2,4}d_{3,4}\neq 0$. 
This exhibits a bad cycle $v_1 \to v_2 \to v_3 \to v_1$, and a defect of type $(4)_A$ in $\Gamma$, contradicting that
$\Gamma$ is defectless.

Therefore, the two zero weight edges are disjoint, say $d_{1,3}=d_{2,4}=0$ and $d_{1,2}d_{1,4}d_{2,3}d_{3,4}\neq 0$.
Now, to avoid a defect of type $(4)_B$, we must have $d_{1,2}d_{3,4} + d_{2,3}d_{1,4} = 0$.
With this restriction on weights, the Table \ref{Case 4,2,1} provides a consistent labeling for $\Gamma$, 
with $\alpha_4 = 0$.

\noindent{(d).} $|E_0| = 3$. \quad Since $\Gamma$ does not have a null vertex, all three edges with zero weight do not share a common vertex. Therefore, up to weighted isomorphism, we assume that either

\begin{enumerate}
\item[(i).] $d_{1,2}=d_{2,3}=d_{3,4}=0$ and $d_{1,3}d_{1,4}d_{2,4} \neq 0$, or
\item[(ii).] $d_{1,2}=d_{2,4}=d_{1,4}=0$ and $d_{1,3}d_{2,3}d_{3,4} \neq 0$. 
\end{enumerate}
The case (i) leads to a defect of type $(4)_A$ in $\Gamma$, with respect to the bad cycle $v_1 \to v_2 \to v_3 \to v_1$. Hence, this case does not occur.

In the case (ii), the Table \ref{Case 4,3,2} provides a consistent labeling for $\Gamma$.

\noindent{(e).} $|E_0| = 4$. \quad In this case, up to weighted isomorphism, we have
$d_{1,2}=d_{2,3}=d_{3,4}=d_{1,4}=0$ and $d_{1,3}d_{2,4} \neq 0$. This leads to a defect of type $(4)_A$ in $\Gamma$, with respect to the bad cycle $v_1 \to v_2 \to v_3 \to v_1$. Hence, this case does not occur.

\noindent{(f).} $|E_0| > 4$. \quad This case cannot occur because the condition $|E_0| > 4$ imposes a null vertex in $\Gamma$. However, by our initial assumption, the graph $\Gamma$ does not contain a null vertex.
\end{proof}

Thus, we have examined the problem of consistent labeling for weighted graphs with four vertices. Extending this analysis to graphs with more vertices requires a formulation of more defects. The weighted graph in Figure \ref{Figure (5)_B} with five vertices does not have a consistent labeling. At the same time, it does not have any of the defects that we have discussed.

In a forthcoming work, we classify all defects in graphs with five vertices.

\begin{figure}[h]
\begin{tikzpicture}
  [scale=0.6,auto=left,every node/.style={circle,fill=blue!20}]
  \node (n3) at (9,1)      {$v_3$};
  \node (n1) at (1,5)      {$v_1$};
  \node (n4) at (1,1)      {$v_4$};
  \node (n2) at (9,5)      {$v_2$}; 
  \node (n5) at (5,9)      {$v_5$}; 
  
  \foreach \from/\to in {n1/n2,n1/n4,n1/n5,n2/n3,n2/n5,n3/n4}
      \draw (\from) -- (\to);
\end{tikzpicture}
\caption{Graph $\Gamma$ with $d_{1,5}d_{2,5}d_{3,4}\ne 0$ and $d_{1,2} = d_{1,4} = d_{2,3} = 0$.}\label{Figure (5)_B}
\end{figure}
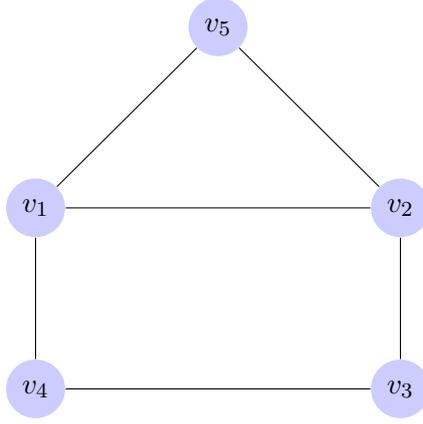



\section{Surjectivity of bilinear maps and Lie bracket}\label{section on bilinear maps and Lie algebra}

Let $F$ be a field and $U$ be a vector space over $F$ whose dimension is at most countable. Let $B: U \times U \to W$ be an alternating bilinear map such that $W$ is the subspace spanned by the set $B(U \times U)$. In this section, we examine the conditions when $W = B(U \times U)$. Our results on labeling of graphs proved in \S\ref{Labeling of graphs with four vertices} play a crucial role in this analysis.

Let $U^\perp := \{v \in U : B(v, v') =0, \text{ for all } \ v' \in U\}$ be the orthogonal complement of $U$ and $\mathcal B_1:= \{v_1, v_2, \dots \}$ be a basis of $U^{\perp}$. We extend $\mathcal B_1$ to a basis $\mathcal{B}:= \mathcal B_1 \cup \mathcal B_2$ of $U$. Let $\mathcal B_2 = \{u_1, u_2, \dots\}$. Then
$$W = {\rm span}(\{B(u_i, u_j) : u_i, u_j \in \mathcal{B}_2, i < j\}).$$ 
Evidently, we can write an arbitrary $w \in W$ as $$w = \sum_{1 \leq i < j \leq r} d_{i,j} B(u_i, u_j); \quad d_{i,j} \in F.$$ 
It's worth noting that the set $\{B(u_i, u_j) : u_i, u_j \in \mathcal{B}_2\}$ need not be a basis of $W$. Thus, for any $w \in W$, there may exist multiple choices for the scalers $d_{i,j}$. 
Let 
$$\mathcal{I} := \{i : d_{i,j}\ne 0 \text{ for some $j$}\} \cup \{j : d_{i,j}\ne 0 \text{ for some $i$}\}$$

Let $|\mathcal{I}| = n$. We permute the indices $\{1,2, \cdots, \}$ so that $\mathcal{I} = \{1, 2, \cdots , n\}$. Now, we consider the weighted graph $\Gamma$ with vertex set $\{v_1, v_2, \dots, v_n\}$, where two vertices
$v_i$ and $v_j$ are adjacent if $B(u_i, u_j) \neq 0$. On the connecting edge $e_{i,j}$, we assign the weight  $d_{i,j}$.

The weighted graph $\Gamma$ with the weight sequence $D = (d_{i,j})_{e_{i,j} \in E}$ depends on the set $\mathcal{B}_2$, the element $w$, and the permutation that sorts out elements of $\mathcal I$ as first $n$ indices. A weight sequence obtained through this process is called a \emph{presentation} of $w$. The weighted graph $\Gamma$ corresponds to a system of balance equations and it is clear that $w$ lies in the image of $B$ if and only if there exists a presentation $D = (d_{i,j})_{e_{i,j} \in E}$ of $w$ such that the corresponding graph $\Gamma$ has a consistent labeling. We refer to \cite[\S5]{KK_2025(1)} for more details.

The following proposition is a direct consequence of Theorem \ref{4 vertices classification}.

\begin{proposition}\label{Elementwise Bilinear map without defects proposition}
Let $F$ be a field and $U, W$ be two vector spaces over $F$ with at most countable dimension. Let $B: U \times U \to W$ be an alternating bilinear map over $F$ such that ${\rm span}(B(U \times U)) = W$. Let $w \in W$ be such that for some presentation $D$ of $w$, the corresponding graph $\Gamma$ has $4$ vertices and is defectless. Then $w \in B(U \times U)$.
\end{proposition}

The following proposition follows directly from Lemmas \ref{defect (m)_A Lemma} -- \ref{defect (4)_C Lemma}.

\begin{proposition}\label{Elementwise Bilinear map defects proposition}
Let $F$ be a field and $U, W$ be two vector spaces over $F$ with at most countable dimension. Let $B: U \times U \to W$ be an alternating bilinear map over $F$ such that ${\rm span}(B(U \times U)) = W$. Let $w \in W$ be such that for each presentation $D$ of $w$, the corresponding weighted graph $\Gamma$ contains a defect of type $(m)_A$, $(4)_B$ or $(4)_C$. Then $w \notin B(U \times U)$.
\end{proposition}

The above propositions have an interesting Corollary.
\begin{corollary}\label{Whole Space Bilinear map defects proposition}
Let $F$ be a field and $U, W$ be two vector spaces over $F$ with at most countable dimension. Let $B: U \times U \to W$ be an alternating bilinear map over $F$ such that $\dim(U / U^{\perp})$ is $4$ and ${\rm span}(B(U \times U)) = W$. Let $w \in W$. Then $w \in B(U \times U)$ if and only if there exists a presentation $D$ of $w$ such that the corresponding weighted graph $\Gamma$ is free from defects of type $(4)_A$, $(4)_B$ and $(4)_C$.
\end{corollary}

The following proposition is an interesting application of our results on labeling for graphs having $4$ vertices.

\begin{proposition}\label{Elememtwise Bilinear map dim 3}
Let $F$ be a field and $U, W$ be two vector spaces over $F$ with at most countable dimension. Let $B: U \times U \to W$ be an alternating bilinear map over $F$ such that ${\rm span}(B(U \times U)) = W$. Let $\mathcal B : =\{u_1, u_2, \cdots\}$ be an ordered basis of $U /U^{\perp}$. Let $\mathcal B_2$ be a finite subset of $\mathcal B$, and $W' := {\rm span}(B(\mathcal B_2 \times \mathcal B_2 ))$. If $\dim(W') \leq 3$, then $W' \subseteq B({\rm span}(\mathcal B_2) \times {\rm span}(\mathcal B_2))$.
\end{proposition}

\begin{proof}
We proceed case by case depending on $\dim(W')$.

\textbf{Case 1}. $\dim(W') = 1$. \quad Let $u_{i_1}, u_{i_2} \in \mathcal B_2$
be such that $B(u_{i_1}, u_{i_2}) \neq 0$. Since $\dim(W') = 1$, the singleton set $\mathcal B_{W'} := \{B(u_{i_1}, u_{i_2})\}$ is a basis of $W'$, and for each $w \in W'$, there exists $a_w \in F$ such that $w = a_w B(u_{i_1}, u_{i_2})$. Thus, $w = B(a_wu_{i_1}, u_{i_2}) \subseteq B({\rm span}(\mathcal B_2) \times {\rm span}(\mathcal B_2))$.

\textbf{Case 2}. $\dim(W') = 2$. \quad Let $u_{i_1}, \dots, u_{i_4} \in \mathcal B$ be such that 
$\mathcal B_{W'} := \{B(u_{i_1}, u_{i_2}), B(u_{i_3}, u_{i_4})\}$ is a basis of $W'$.
We consider the linear combination $w = a_1 B(u_{i_1}, u_{i_2}) + a_2 B(u_{i_3}, u_{i_4}) \in W'$, where $a_1, a_2, \in F$. We have the following two possibilities.

\begin{enumerate}
\item[(a).] $i_1, i_2, i_3$ and $i_4$ are not all distinct. \quad It is enough to examine the case when $i_1 = i_3$ or $i_2 = i_4$. If $i_1 = i_3$, then $w = B(u_{i_1}, a_1 u_{i_2} + a_2 u_{i_4}) \in B({\rm span}(\mathcal B_2) \times {\rm span}(\mathcal B_2))$. The case $i_2 = i_4$ is similar.

\item[(b).] $i_1, i_2, i_3$ and $i_4$ are all distinct. \quad We reenumerate the elements of $\mathcal B$ so that $i_1 = 1, i_2 = 2, i_3 = 3$, and $i_4 = 4$. Let $D = (d_{i,j})_{e_{i,j} \in E}$ be a presentation of $w$ and $\Gamma$ be the corresponding weighted graph in $4$ vertices containing edges $e_{1,2}$ and $e_{3,4}$. Let $E$ be the edge set for $\Gamma$. If $E = \{e_{1,2}, e_{3,4}\}$, then $B(u_i, u_j) = 0$, whenever $i \in \{1,2\}$ and $j \in \{3,4\}$. In that case, 
$$w = a_1 B(u_{1}, u_{2}) + a_2 B(u_{3}, u_{4}) = B(a_1u_1 + a_2u_3, u_2+u_4),$$
which clearly lies in $B({\rm span}(\mathcal B_2) \times {\rm span}(\mathcal B_2))$. 
Thus we assume that $\Gamma$ has an edge $e_{i,j}$ with $i \in \{1,2\}$ and $j \in \{3,4\}$. Without loss of generality, let it be $e_{1,3}$. Then $B(u_1, u_3) \ne 0$, and hence either $\{B(u_{1}, u_{2}), B(u_{1}, u_{3})\}$ or $\{B(u_{1}, u_{3}), B(u_{3}, u_{4})\}$ is a basis of $W'$. Now, the case 2(a) takes over and we have
$w \in B({\rm span}(\mathcal B_2) \times {\rm span}(\mathcal B_2))$.
\end{enumerate}

\textbf{Case 3}. $\dim(W') = 3$. \quad Let $u_{i_1},\dots u_{i_6} \in \mathcal B_2$ be such that
$$\mathcal B_{W'} := \{B(u_{i_1}, u_{i_2}), B(u_{i_3}, u_{i_4}), B(u_{i_5}, u_{i_6})\}$$
is a basis of $W'$.
Let $w = a_1 B(u_{i_1}, u_{i_2}) + a_2 B(u_{i_3}, u_{i_4}) + a_3 B(u_{i_5}, u_{i_6}) \in W'$. We first assume that $\{i_{1}, i_{2}\} \cap \{i_{3}, i_{4}\} \cap \{i_5, i_6\} \ne \phi$. Without loss of generality, let $1 =i_{1} = i_{3} = i_5$. Then $w = a_1 B(u_{i_1}, u_{i_2}) + a_2 B(u_{i_3}, u_{i_4}) + a_3 B(u_{i_5}, u_{i_6}) = B(u_{i_1}, a_1 u_{i_2} + a_2 u_{i_4} + a_3 u_{i_6})$, which clearly lies in $B({\rm span}(\mathcal B_2) \times {\rm span}(\mathcal B_2))$. Therefore, let $\{i_{1}, i_{2}\} \cap \{i_{3}, i_{4}\} \cap \{i_5, i_6\} = \phi$. 

It is enough to show that for some presentation of $w$, the corresponding weighted graph has consistent labeling.
Let $D = (d_{i,j})_{e_{i,j} \in E}$ be a presentation of $w$ with respect to the basis $\mathcal B_{W'}$, and $\Gamma$ be the corresponding weighted graph. If $\Gamma$ is a tree, then by Lemma $3.1$ of \cite{KK_2025(1)}, it admits a consistent labeling and hence $w \in B({\rm span}(\mathcal B_2) \times {\rm span}(\mathcal B_2))$. Furthermore, if $\Gamma$ is a cycle, then by Corollary $3.6$ and Lemma $3.7$ of \cite{KK_2025(1)}, it admits a consistent labeling, and again $w \in B({\rm span}(\mathcal B_2) \times {\rm span}(\mathcal B_2))$. So, we assume that $\Gamma$ is neither a tree, nor a cycle. Let $V$ be the vertex set of $\Gamma$. If $|V| \leq  3$, then $\Gamma$ is either a tree or a triangular cycle, and thus we assume that $4 \leq |V| < 6$. We consider the following subcases, depending on the size of $V$.

\begin{enumerate}
\item[(a).] $|V| = 4$. \quad If $\Gamma$ is a connected graph in $4$ vertices that is neither a tree nor a cycle, then it is isomorphic to either $\Gamma_1$, $\Gamma_2$ or $\Gamma_3$ in Figure \ref{4 vertices graphs table}. We analyze
these scenarios one by one.
\begin{enumerate}
\item[(i).] $\Gamma \simeq \Gamma_1$. \quad Since $\Gamma$ is free from null vertices, $d_{3,4} \ne 0$. Consequently,
$$(3,4) \in \{({i_1, i_2}), ({i_3, i_4}), ({i_5, i_6})\}.$$
Since$\{i_{1}, i_{2}\} \cap \{i_{3}, i_{4}\} \cap \{i_5, i_6\} = \phi$, upto a permutation of indices,
any choice of $(i_{1}, i_{2})$, $(i_{3}, i_{4})$ and $(i_{5}, i_{6})$ is equivalent to the choice $(i_{1}, i_{2}) = (3,4)$, $(i_{3}, i_{4}) = (1,3)$ and $(i_{5}, i_{6}) = (1,2)$. 
Recalling that $w = a_1 B(u_{i_1}, u_{i_2}) + a_2 B(u_{i_3}, u_{i_4}) + a_3 B(u_{i_5}, u_{i_6})$, the weights on $\Gamma$ are 
$$d_{3,4} = a_1, \quad d_{1,3} = a_2, \quad  d_{1,2} = a_3, \quad d_{2,3} = 0.$$
Assuming that $\Gamma$ does not have consistent labeling, we exhibit another presentation of $w$ whose corresponding weighted graph has a consistent labeling. Since $\Gamma$ does not have a consistent labeling, by Theorem \ref{4 vertices classification}, $d_{1,3} = a_2 = 0$, that is, 
$$w = a_1 B(u_{3}, u_{4}) + a_3 B(u_{1}, u_{2}).$$

Since $\{B(u_3, u_4), B(u_1, u_3), B(u_1, u_2)\}$ is a basis of $W'$, we write
$$B(u_2, u_3) = b_{3,4} B(u_3, u_4) + b_{1,3} B(u_1, u_3) + b_{1,2} B(u_1, u_2)$$

We assume that $b_{3,4} = 0$; otherwise the set $\{B(u_2, u_3), B(u_1, u_3), B(u_1, u_2)\}$ would form a basis of
$W'$, yielding a presentation of $w$ whose associated weighted graph has three vertices, a case that we have already resolved. \\

Suppose, $b_{1,2} \neq 0$. Then $\{B(u_3, u_4), B(u_1, u_3), B(u_2, u_3)\}$ is a basis of $W'$, and
$$
w = a_1 B(u_3, u_4) + a_3 b_{1,2}^{-1} (B(u_2,u_3) - b_{1,3} B(u_1,u_3)).
$$
yields a presentation of $w$ whose associated weighted graph has $0$ weight on the edge $e_{1,2}$. 
Hence, by Theorem \ref{4 vertices classification} it admits a consistent labeling. \\

Therefore, we assume $b_{1,2} = 0$ as well. In this case, $0 \ne B(u_2, u_3) = b_{1,3} B(u_1, u_3)$. 
Thus, $b_{1,3} \ne 0$, and
\begin{align*}
w &= a_1 B(u_{3}, u_{4}) + a_3 B(u_{1}, u_{2}) \\
&= a_1 B(u_{3}, u_{4}) + a_3 B(u_{1}, u_{2}) + B(u_2, u_3) - b_{1,3} B(u_1, u_3).
\end{align*}
If $a_1 =0$ or $a_3 =0$, then $w \in B({\rm span}(\mathcal B_2) \times {\rm span}(\mathcal B_2))$, thus we assume that
$a_1$ and $a_3$ are both nonzero. This yields a presentation $D' = (d_{i,j}')_{e_{i,j}\in E}$ of $w$, where 
$d'_{3,4} = a_1 \ne 0$, $d'_{1,2} = a_3 \ne 0$, $d'_{1,3} = -b_{1,3} \neq 0$ and $d'_{2,3} = 1$. Now, by Theorem \ref{4 vertices classification}, the weighted graph associated to $D'$ has consistent labeling.

\item[(ii).]  $\Gamma \simeq \Gamma_2$. \quad We extend $0 \neq B(u_1, u_3) \in W'$ to a basis $\mathcal B_{W'}$ of $W'$.
Since we have already resolved the case where the weighted graph corresponding to a presentation of a vector $w\in W$ is a cycle, and $\{i_{1},i_{2}\} \cap \{i_{3},i_{4}\} \cap \{i_5, i_6\} = \phi$, we assume 
$$\mathcal B_{W'} = \{B(u_1, u_3), B(u_1, u_2), B(u_3, u_4)\}.$$
and write
$$w = a_1 B(u_{1}, u_{3}) + a_2 B(u_{1}, u_{2}) + a_3 B(u_{3}, u_{4}).$$
Thus, the weights on the graph corresponding to this presentation are
$$d_{1,3} = a_1, d_{1,2} = a_2, d_{3,4} = a_3, d_{1,4} = 0, d_{2,3} =0.$$

We assume that $a_1 = 0$, otherwise by Theorem \ref{4 vertices classification}, $\Gamma$ has a consistent labeling. \\

Since $\{B(u_1, u_3), B(u_1, u_2), B(u_3, u_4)\}$ is a basis of $W'$, we write
$$B(u_{1}, u_{4}) = b_{1,3} B(u_{1}, u_{3}) + b_{1,2} B(u_{1}, u_{2}) + b_{3,4} B(u_{3}, u_{4}).$$

We assume that $b_{1,2} = 0$; otherwise the set $\{B(u_1, u_3), B(u_1, u_4), B(u_3, u_4)\}$ would form a basis of
$W'$, yielding a presentation of $w$ whose associated weighted graph has three vertices, and we have already resolved that case. \\

Suppose, $b_{3,4} \neq 0$. Then $\{B(u_1, u_3), B(u_1, u_2), B(u_1, u_4)\}$ is a basis of $W'$ such that $\{i_{1}, i_{2}\} \cap \{i_{3}, i_{4}\} \cap \{i_5, i_6\} = 1$. This contradicts our assumption 
$\{i_{1}, i_{2}\} \cap \{i_{3}, i_{4}\} \cap \{i_5, i_6\} = \phi$. Therefore, we assume $b_{3,4} = 0$ as well. In this case, $B(u_{1}, u_{4}) = b_{1,3} B(u_{1}, u_{3})$. Thus, $b_{1,3} \neq 0$ and
\begin{align*}
w &= a_2 B(u_{1}, u_{2}) + a_3 B(u_{3}, u_{4}) \\
&= a_2 B(u_{1}, u_{2}) + a_3 B(u_{3}, u_{4}) + b_{1,3} B(u_{1}, u_{3}) - B(u_{1}, u_{4}).
\end{align*}

This yields a presentation $D' = (d_{i,j}')_{e_{i,j}\in E}$ of $w$, with $d'_{1,3} = b_{1,3} \ne 0$, and hence by Theorem \ref{4 vertices classification}, the weighted graph associated to $D'$ has consistent labeling.

\item[(iii).]  $\Gamma \simeq \Gamma_3$. \quad We extend $0 \neq B(u_1, u_2) \in W'$ to a basis $\mathcal B_{W'}$ of $W'$.
Since we have already resolved the case where the weighted graph corresponding to a presentation of a vector $w\in W$ is a cycle, and $\{i_{1}, i_{2}\} \cap \{i_{3}, i_{4}\} \cap \{i_5, i_6\} = \phi$, we assume by a permutation of indices, if required, that
$\mathcal B_{W'} = \{B(u_1, u_2), B(u_2, u_3), B(u_3, u_4)\}$. \\

We write 
$w = a_1 B(u_{1}, u_{2}) + a_2 B(u_{2}, u_{3}) + a_3 B(u_{3}, u_{4})$, and assume that $a_1 \ne 0$,  $a_3 \ne 0$; otherwise the corresponding weighted graph would have at most $3$ vertices, and we have resolved that case. Now, write 
$$B(u_{1}, u_{3}) = b_{1,2} B(u_{1}, u_{2}) + b_{2,3} B(u_{2}, u_{3}) + b_{3,4} B(u_{3}, u_{4}).$$ 

By an argument analogous to (i) and (ii) above, we may assume that $b_{1,2} = 0$ and $b_{3,4} = 0$, that leads to
$0 \ne B(u_{1}, u_{3}) = b_{2,3} B(u_{2}, u_{3})$, which means $b_{2,3} \ne 0$. 

Similarly, writing $B(u_{2}, u_{4})$ as a linear combination in $\mathcal B_{W'}$, there exists $c_{2,3} \in F$
such that $0 \ne B(u_{2}, u_{4}) = c_{2,3} B(u_{2}, u_{3})$. Therefore,
\begin{align*}
w &= a_1 B(u_{1}, u_{2}) + a_2 B(u_{2}, u_{3}) + a_3 B(u_{3}, u_{4}) \\
&= a_1 B(u_{1}, u_{2}) + (a_2 - c_{2,3} - b_{2,3}a_1a_3) B(u_{2}, u_{3}) \\
&\quad\quad\quad\quad\quad\quad+ a_1a_3B(u_{1}, u_{3}) + B(u_{2}, u_{4}) + a_3 B(u_{3}, u_{4}).    
\end{align*}

This yields a presentation $D' = (d_{i,j}')_{e_{i,j}\in E}$ of $w$, with $d'_{1,4} = 0$, $d'_{1,2}=a_1 \neq 0$, $d'_{1,3} = a_1 a_3 \neq 0$, $d'_{2,3} = a_2 - c_{2,3} - b_{2,3}a_1a_3$, $d'_{2,4} = 1$, $d'_{3,4} = a_3 \neq 0$.
Note that it satisfies $d'_{1,2} d'_{3,4} = d'_{1,3} d'_{2,4}$. Hence, the weighted graph corresponding to this presentation is free from defects, and by Theorem \ref{4 vertices classification}, it has a consistent labeling.
\end{enumerate}

\item[(b).] $|V| =5$. \quad Since $\mathcal B_{W'} := \{B(u_{i_1}, u_{i_2}), B(u_{i_3}, u_{i_4}), B(u_{i_5}, u_{i_6})\}$ is a basis of $W'$, where $u_{i_1},\dots, u_{i_6} \in \mathcal B_2$ and $|V| =5$, exactly two indices among ${i_1},\dots, {i_6}$ are the same. A reenumeration of indices allows us to assume that $\mathcal B_{W'} := \{B(u_{1}, u_{2}), B(u_{2}, u_{3}), B(u_{4}, u_{5})\}$. Thus,
$$w = a_1B(u_{1},u_{2}) + a_2 B(u_{2}, u_{3})  + a_3 B(u_{4}, u_{5}).$$

If $B(u_{2},u_{4}) \ne 0$, then it can be extended to a basis of $\mathcal B_W'$ of $W'$ by adding two suitable vectors from $\{B(u_{1}, u_{2}), B(u_{2}, u_{3}), B(u_{4}, u_{5})\}$. The new basis yields a presentation of $D' = (d_{i,j}')_{e_{i,j}\in E}$ of $w$ so that the corresponding weighted graph has 4 vertices, and we are done by the previous case. Thus, we assume that $B(u_{2},u_{4}) = 0$. By a similar argument, we assume $B(u_{2},u_{5}) = 0$, and therefore $\Gamma$ may be assumed to be a subgraph of the graph $\Gamma_7$ shown in Figure \ref{dim 3 Figure}.\\

\begin{figure}[h]
\begin{tabular}{ccccc}
\begin{tikzpicture}
  [scale=0.6,auto=left,every node/.style={circle,fill=blue!20}]
  \node (n3) at (9,1)      {$v_5$};
  \node (n1) at (1,5)      {$v_1$};
  \node (n4) at (1,1)      {$v_4$};
  \node (n2) at (9,5)      {$v_3$}; 
  \node (n5) at (5,9)      {$v_2$}; 
  
  \foreach \from/\to in {n1/n2,n1/n3,n1/n4,n1/n5,n2/n3,n2/n4,n2/n5,n3/n4}
      \draw (\from) -- (\to);
\end{tikzpicture}
\hspace{2 cm}
&
\begin{tikzpicture}
  [scale=0.6,auto=left,every node/.style={circle,fill=blue!20}]
  \node (n3) at (9,5)      {$v_3$};
  \node (n1) at (5,9)      {$v_2$};
  \node (n4) at (1,5)      {$v_1$};
  \node (n2) at (1,1)     {$v_4$}; 
  \node (n5) at (9,1)     {$v_5$};  
  
  \foreach \from/\to in {n2/n3,n1/n3,n1/n4,n3/n4,n2/n5,n3/n5}
      \draw (\from) -- (\to);
\end{tikzpicture}\\

$\Gamma_7$ & $\Gamma_8$ &  \\
\end{tabular}
\caption{}
\label{dim 3 Figure}
\end{figure}

Now, write
$$B(u_{1}, u_{4}) = b_{1,2} B(u_{1}, u_{2}) + b_{2,3} B(u_{2}, u_{3}) + b_{4,5} B(u_{4}, u_{5}).$$
If $b_{2,3} \ne 0 $, then $\{B(u_1, u_2), B(u_1, u_4), B(u_4, u_5)\}$ is a basis of $W'$ and we are reduced to the previous case of $4$ vertices. Thus, we assume $b_{2,3} = 0$. By a similar argument, we assume $b_{4,5} = 0$.
Hence, $B(u_{1}, u_{4}) = b_{1,2} B(u_{1}, u_{2})$. \\

Similarly, by expanding $B(u_{1}, u_{5})$ as basis vectors, we obtain a scalar $c$ such that $B(u_{1}, u_{5}) =  c B(u_{1}, u_{2})$. Arguing similarly, we obtain scalars $c',c''$ such that $B(u_{3}, u_{4}) = c' B(u_{2}, u_{3})$ and $B(u_{3}, u_{5}) = c'' B(u_{2}, u_{3})$. \\

Now, if $(b_{1,2}, c) \ne (0,0)$ and $(c', c'') \ne (0,0)$, say $b_{1,2 }\ne 0$ and $c'' \ne 0$, then the set $\{B(u_1, u_4), B(u_3, u_5), B(u_4, u_5)\}$ is a basis of $W'$ and we are reduced to the earlier case of $4$ vertices.
So, we assume that either $(b_{1,2}, c) = (0,0)$ or $(c', c'') = (0,0)$. 
If $(c', c'') = (0,0)$, then $\Gamma$ is disconnected, and thus it admits a consistent labeling because each of its connected components does so. Thus, we assume that $(c', c'') \ne (0,0)$ and $(b_{1,2}, c) = (0,0)$. Hence $B(u_{1}, u_{4}) = 0 = B(u_{1}, u_{5})$, and $\Gamma$ is a subgraph of the graph $\Gamma_8$ in Figure \ref{dim 3 Figure}.  Without loss of generality, let $c' \ne 0$. Then $\{B(u_{1}, u_{2}), B(u_{3}, u_{4}), B(u_{4}, u_{5}\}$ is a basis of $W'$. Writing $B(u_{1}, u_{3})$ as a linear combination in $\{B(u_{1}, u_{2}), B(u_{3}, u_{4}), B(u_{4}, u_{5}\}$, and following arguments similar to those used in proving $B(u_{1}, u_{4}) = b_{1,2} B(u_{1}, u_{2})$ above, we obtain that $B(u_{1}, u_{3}) = \lambda_1 B(u_{3}, u_{4})$ for some $\lambda_1 \in F$. Hence, $B(u_{1}, u_{3}) = \lambda B(u_{2}, u_{3})$, where $\lambda = \lambda_1 c'$. Thus,
\begin{align*}
w &= a_1B(u_{1},u_{2}) + a_2 B(u_{2}, u_{3}) + a_3 B(u_{4}, u_{5}) \\
&= a_1B(u_{1},u_{2}) + (a_2 -c'-\lambda) B(u_{2}, u_{3}) + B(u_{1}, u_{3}) +  B(u_{3}, u_{4}) +  a_3 B(u_{4}, u_{5}).  
\end{align*}

This yields a presentation $D' = (d_{i,j}')_{e_{i,j}\in E}$ of $w$, with 
$d'_{1,3} = d'_{3,4} = 1$, $d'_{1,2}=a_1$, $d'_{2,3} = a_2 -c'- \lambda$, $d'_{4,5} = a_3 $ and $d'_{3,5} = 0$. Then Table \ref{dim 3 Table}, provides a consistent labeling for the corresponding weighted graph. \\

\begin{table}[h]
\parbox{.6\linewidth}{
\centering
\begin{tabular}{|c|c|c|}
 \hline
 $i$ &  $\alpha_i$ &$\beta_i$\\
 \hline
 $1$ & $\beta_3^{-1}$ & $\beta_1 \in F$\\
 \hline
 $2$ & $(a_2 -c'- \lambda)\beta_3^{-1}$ & $a_1\beta_3 + (a_2 -c'- \lambda)\beta_1$\\
 \hline
$3$ & $0$ & $\beta_3 \in F^{\times}$\\
 \hline
$4$ & $-\beta_3^{-1}$ & $\beta_4 \in F$\\
 \hline
$5$ & $0$ & $-a_3\beta_3$\\
 \hline
\end{tabular} 
\caption{}
\label{dim 3 Table}
}
\end{table}

\item[(c).] $|V| = 6$. \quad After a reenumeratuion of indices, if required, we assume that $B_{W'} := \{B(u_{1}, u_{2}), B(u_{3}, u_{4}), B(u_{5}, u_{6})\}$ is a basis of $W'$. We denote $w_1 := B(u_{1}, u_{2})$, $w_2 := B(u_{3}, u_{4})$, $w_3 := B(u_{5}, u_{6})$. Suppose $B(u_1, u_4) \neq 0$, and write
$$B(u_{1}, u_{4}) = b_1 w_1 + b_2 w_2 + b_3 w_3.$$

If $b_1 \neq 0 $, then $\{B(u_1, u_4), w_2, w_3\}$ is a basis of $W'$ and hence there exists a presentation of $w$ whose corresponding weighted graph has $5$ vertices, and we are reduced to one of the previous cases. Same argument works when $b_2 \neq 0$ or $b_3 \neq 0$. \\

Thus, we assume that $B(u_{1}, u_{4}) = 0$. In fact, a similar argument shows that $B(u_i, u_j) = 0$, whenever $1 \leq i < j \leq 6$ and $(i,j) \notin \{(1,2), (3,4), (5,6)\}$. Hence, the graph $\Gamma$ is disjoint union of at most three edges; $e_{1,2}$, $e_{3,4}$ and $e_{5,6}$, which clearly admits a consistent labeling.   
\end{enumerate}
\end{proof}

The following corollary follows directly from Proposition \ref{Elememtwise Bilinear map dim 3} 

\begin{corollary}\label{Whole space Bilinear dim 3} 
Let $F$ be a field and $U, W$ be two vector spaces over $F$ whose dimension is at most countable. Let $B: U \times U \to W$ be an alternating bilinear map over $F$ such that ${\rm span}(B(U \times U)) = W$. If $\dim(W) \leq 3$, then $W = B(U \times U)$.
\end{corollary}

For a Lie algebra $L$ over a field $F$, the Lie bracket is an alternating bilinear map. Thus, in this case $U:= L$; $W:= L'$, the derived Lie subalgebra of $L$; $B = [~,~]$, the Lie bracket of $L$; $U^\perp = Z(L)$, the center of $L$; and $\mathcal{B}_L := \mathcal B_1 \cup \mathcal B_2$, the vector space basis of $L$; where $\mathcal B_1 $ is a basis of $Z(L)$ and $\mathcal B_2 = \{u_1, u_2 \cdots \}$. The following results follow directly from the above results on alternating bilinear maps.

\begin{theorem}\label{Elementwise Lie algebra without defects theorem}
Let $F$ be a field and $L$ be a Lie algebra over $F$ with at most countable dimension. Let $x \in L'$ be such that for some presentation $D$ of $x$, the corresponding weighted graph $\Gamma$ has four vertices and is defectless. Then $x \in [L,L]$.
\end{theorem}

\begin{theorem}\label{Elementwise Lie algebra defects theorem}
Let $F$ be a field and $L$ be a Lie algebra over $F$ with at most countable dimension. Let $x \in L'$ be such that for each presentation $D$ of $x$, the corresponding weighted graph $\Gamma$ contains a defect of type $(m)_A$, $(4)_B$ or $(4)_C$. Then $x \notin [L,L]$.
\end{theorem}

This leads us to the following corollary.

\begin{corollary}\label{Whole Space Lie algebra defects}
Let $F$ be a field and $L$ be a Lie algebra over $F$ with at most countable dimension. Let $\dim(L /Z(L)) = 4$. Let $x \in L'$. Then $x \in [L,L]$ if and only if there exists a presentation $D$ of $x$ such that the correspoding weighted graph $\Gamma$ is defectless.
\end{corollary}

From Proposition \ref{Elememtwise Bilinear map dim 3}, we obtained the following.

\begin{proposition}\label{Elememtwise Lie algebra dim 3}
Let $F$ be a field and $L$ be a Lie algebra over $F$ with at most countable dimension. Let $\mathcal B_L : =\{u_1, u_2, \cdots\}$ be an ordered basis of $L / Z(L)$, $\mathcal B_2$ be a finite subset of $\mathcal B_L$, and $W' := {\rm span}([\mathcal B_2,\mathcal B_2])$. If $\dim(W') \leq 3$, then $W' \subseteq [{\rm span}(\mathcal B_2),{\rm span}(\mathcal B_2)] \subseteq [L,L]$.
\end{proposition}

\begin{theorem}\label{Whole space Lie algebra dim 3} 
Let $F$ be a field and $L$ be a Lie algebra over $F$ with at most countable dimension. If $\dim(L') \leq 3$, then $L'=[L,L]$.    
\end{theorem}

If $\dim(L') > 3$, then Theorem \ref{Elementwise Lie algebra defects theorem} can be used to obtain infinitely many examples with $[L,L] \ne L'$. We present one such example.
\begin{example}
Let $F$ be an arbitrary field and $L$ be the Lie algebra over $F$ defined by the presentation:
\begin{align*}
\langle x_1,x_2,\dots, x_8: [x_1,x_2] & = x_5, [x_1,x_3] = x_6, [x_1,x_4] = x_7,  [x_2,x_3] = x_8,\\ 
&[x_2,x_4] = [x_3,x_4] = 0, [x_i,x_j] = 0,\text{ if } i \text{ or } j \in \{5,6,7,8\}\rangle. 
\end{align*}

Then $L' = Z(L) = \langle x_5, x_6, x_7, x_8\rangle$ is four dimensional and $L$ is a nilpotent Lie algebra of class $2$. We claim that the element $x = x_7 + x_8 \in L'$ is not a commutator. Assuming $x \in [L,L]$, there exist $a_i, b_i \in F$;  $i \in \{1,2,3,4\}$ such that 
\begin{align*}   
x_7 + x_8 & = \left[\sum_{i=1}^{4} a_i x_i, \sum_{i=1}^{4} b_i x_i \right] \\ 
& = (a_1 b_2 - a_2 b_1)[x_1,x_2] + (a_1 b_3 - a_3 b_1)[x_1,x_3] \\
& \hspace{3.85 cm} +(a_1 b_4 - a_4 b_1)[x_1,x_4] + (a_2 b_3 - a_3 b_2)[x_2,x_3]\\
& = (a_1 b_2 - a_2 b_1)x_5 + (a_1 b_3 - a_3 b_1)x_6 + (a_1 b_4 - a_4 b_1)x_7 + (a_2 b_3 - a_3 b_2)x_8.
\end{align*} Since $\{x_5,x_6,x_7,x_8\}$ is a linearly independent set in $L$, the above equality enforces a consistent labeling on the weighted graph $\Gamma = (V,E)$, where 
$V = \{v_1, v_2, v_3, v_4\}, E = \{e_{1,2},e_{1,3}, e_{1,4}, e_{2,3}\}$, and the weights are
$d_{1,2} = 0 $, $d_{1,3} = 0$, $d_{1,4} = 1 $ and $d_{2,3} = 1$.
Thus $\Gamma$ is weighted isomorphic to the graph in Figure \ref{defect (4)_A}, that carries the defect of type $(4)_A$. 
This is a contradition to Theorem \ref{Elementwise Lie algebra defects theorem}. Thus $x \notin [L,L]$, and hence $L' \ne [L,L]$.
\end{example}


\section{Commutators and Commutator Subgroup of Nilpotent Groups of Class 2}\label{section for results on p-groups}
 
Let $G$ be a finite nilpotent $p$-group of class $2$ and exponent $p$. Let $Z(G)$ be its center and $G'$ be its commutator subgroup. Let $K(G)$ denote the set of commutators in $G$.
Let $g_1, g_2, \dots, g_m \in G$ be such that $B_G := \{g_i Z(G) : 1 \leq i \leq m\}$ is a generating set of the factor group $G/Z(G)$. Let $g, h \in G$ be such that $g = \prod_{i=1}^{m} g_i^{\alpha_i}z_1$ and $h = \prod_{i=1}^{m} g_i^{\beta_i}z_2$, where $z_1, z_2 \in Z(G)$ and $\alpha_i, \beta_i \in \mathbb{Z}$ for $1 \leq i \leq m$. Then

$$[g, h] = \prod_{1 \leq i<j \leq m} \left[g_i, g_j\right]^{\alpha_i \beta_j - \alpha_j \beta_i}.$$

Thus, the group $G'$ is generated by the set $\{\left[g_i,g_j\right] : 1 \leq i < j \leq m\}$. Let $g \in G'$. Then there exist $d_{i,j} \in \mathbb{Z}$ such that
$$\prod_{1 \leq i<j \leq m} [g_i, g_j]^{d_{i,j}}.$$ The choice of integers $d_{i,j}$'s may not be unique. 

Let 
$$\mathcal{I} := \{i : d_{i,j}\ne 0 \text{ for some $j$}\} \cup \{j : d_{i,j}\ne 0 \text{ for some $i$}\}.$$

Let $|\mathcal{I}| = n$. We permute the indices $\{1,2, \cdots, m\}$ so that $\mathcal{I} = \{1, 2, \cdots , n\}$. Now, we consider the weighted graph $\Gamma$ with vertex set $\{v_1, v_2, \dots, v_n\}$, where two vertices
$v_i$ and $v_j$ are adjacent if $[g_i, g_j] \neq 1$. On the connecting edge $e_{i,j}$, we assign the weight $d_{i,j}$
\textit{modulo} $p$. Thus, $d_{i,j}$ are regarded as elements of the finite field $\mathbb F_p$ with $p$ elements.

The weighted graph $\Gamma$ with the weight sequence $D = (d_{i,j})_{e_{i,j} \in E}$ corresponds to a system of balance equations: $x_i y_j - x_j y_i = d_{i,j}$, as $e_{i,j} \in E$. The solutions \emph{modulo} $p$ of these equations can be used to write $g$ as a commutator. It is clear that $g$ is a commutator if and only if there exists a presentation $D = (d_{i,j})_{e_{i,j} \in E}$ of $g$ such that the corresponding graph $\Gamma$ has a consistent labeling. Note that weight sequence $D$ depends on the generating set $B_G$, the element $g$, the choice of $d_{i,j}$, and the permutation that sorts out elements of $\mathcal I$ as first $n$ indices. A weight sequence obtained through this process is called a \emph{presentation} of $g$. 

The following theorem follows directly from Theorem \ref{4 vertices classification}.

\begin{theorem}\label{Elementwise group without defects theorem}
Let $G$ be a $p$-group of nilpotency class $2$ and exponent $p$. Let $g \in G'$ be such that for some presentation of $g$, the corresponding weighted graph $\Gamma$ has four vertices and is defectless. Then $g \in K(G)$.
\end{theorem}

The following theorem is a consequence of Lemma \ref{defect (m)_A Lemma}-- \ref{defect (4)_C Lemma}.

\begin{theorem}\label{Elementwise group defects theorem}
Let $G$ be a $p$-group of nilpotency class $2$. Let $g \in G'$ be such that for all presentations 
of $g$, the corresponding weighted graph $\Gamma$ contains a defect of type $(m)_A$, $(4)_B$ or $(4)_C$. Then $g \notin K(G)$.
\end{theorem}

Note that Theorems \ref{Elementwise group without defects theorem} and \ref{Elementwise group defects theorem} hold for infinite $p$-groups of nilpotency class $2$ as well. The following Corollary is a direct consequence of Theorems \ref{Elementwise group without defects theorem} and \ref{Elementwise group defects theorem}.

\begin{corollary}\label{Whole group defects theorem}
Let $G$ be a $p$-group of nilpotency class $2$. Let the factor group $G /Z(G)$ be minimally generated by at most $4$ elements. Let $g \in G'$. Then $g \in K(G)$ if and only if there exists a presentation of $g$ such that the corresponding weighted graph $\Gamma$ is free from defects of type $(4)_A$, $(4)_B$ and $(4)_C$.
\end{corollary}

\bibliographystyle{amsalpha}
\bibliography{word-maps}

\end{document}